\documentclass[11pt,a4paper]{article}

\usepackage{comment} 
\usepackage{graphicx}
\usepackage{amsmath}
\usepackage{amsfonts}
\usepackage{amssymb}
\usepackage{enumerate}
\usepackage{lscape}
\usepackage{longtable}
\usepackage{rotating}
\usepackage{multirow}
\usepackage{color}
\usepackage{url}
\usepackage{subfigure}
\usepackage{rotating}
\usepackage{comment}
\usepackage{chngcntr}
 \usepackage{calc}
 \usepackage[framemethod=tikz]{mdframed}

\newtheorem{theorem}{Theorem}[section]

\usepackage[ruled,vlined,noline,linesnumbered]{algorithm2e}

\usepackage{titlesec}
\usepackage[titletoc]{appendix}
\usepackage{varwidth}

\newtheorem{corollary}{Corollary}[section]

\newtheorem{lemma}{Lemma}[section]

\newtheorem{proposition}{Proposition}[section]
\newtheorem{remark}{Remark}[section]

\newtheorem{assumption}{Assumption}[section]

\newenvironment{proof}[1][Proof]{\textbf{#1.} }{\ \rule{0.5em}{0.5em} \vspace{1ex}}
\newtheorem{subassumption}{Assumption\normalfont}[assumption]

\newenvironment{assumption*}
{\ifnum\value{subassumption}=0 \stepcounter{assumption}\fi\subassumption}
{\endsubassumption}
\newenvironment{assumption+}[1]
{\renewcommand{\thesubassumption}{#1}\subassumption}
{\endsubassumption}

\usepackage{algorithmic}

\def\real{\mathbb{R}}
\def\R{\mathbb{R}}

\def\flow{f_{\mbox{\rm\scriptsize low}}}

\def\ny{n_y}
\def\nx{n_x}

\DeclareMathOperator{\argmin}{argmin}

\newcommand{\n}[1]{\|#1\|}
\newcommand{\Alpha}{\alpha_{\min}}
\newcommand{\barAlpha}{\bar{\alpha}_{\min}}

\DeclareMathOperator{\conv}{conv}

\DeclareMathOperator{\cm}{cm}

\title{Inexact Direct-Search Methods for Bilevel Optimization Problems}
\author{}

\newcommand{\ftnote}[1]{\textcolor{blue}{[ $\diamondsuit$ #1]}}
\newcommand{\commentdf}[1]{\textcolor{blue}{#1}}

\begin{document}

\author{	
	Youssef Diouane\thanks{Department of Mathematics and Industrial Engineering, Polytechnique Montréal, Montréal, QC, Canada.(\tt{youssef.diouane@polymtl.ca})}
	\and
	Vyacheslav Kungurtsev\thanks{Department of Computer Science,
		Czech Technical University, Czech Republic. 
  (\tt{kunguvya@fel.cvut.cz})
	}
	\and
	Francesco~Rinaldi\thanks{Dipartimento di Matematica ``Tullio Levi-Civita'', Universit\`a
		di Padova, Italy.
		(\tt{rinaldi@math.unipd.it})}
	\and
	Damiano Zeffiro\thanks{Dipartimento di Matematica ``Tullio Levi-Civita'', Universit\`a di Padova, Italy.
		({\tt zeffiro@math.unipd.it})}   
}

\maketitle

\maketitle

\begin{abstract}
	In this work, we introduce new direct-search schemes for the solution of bilevel optimization (BO) problems. Our methods rely on a fixed accuracy blackbox oracle for the lower-level problem, and deal both with smooth and potentially nonsmooth true objectives. We thus analyze for the first time in the literature direct-search schemes in these settings, giving convergence guarantees to approximate stationary points, as well as complexity bounds in the smooth case. We also propose the first adaptation of mesh adaptive direct-search schemes for BO. Some preliminary numerical results on a standard set of bilevel optimization problems show the effectiveness of our new approaches.  
\end{abstract}

\section{Introduction}

Bilevel optimization (see, e.g., \cite{beck2021gentle,colson2007overview,dempe2002foundations,dempe2020bilevel,kleinert2021survey} and references therein for a complete overview on the topic) has been subject of increasing interest, thanks to its  application to hyperparameter tuning for machine learning algorithms and meta-learning (see, e.g.,~\cite{franceschi2018bilevel} and references therein).
In this work, we are interested in the following bilevel optimization problem
\begin{equation}\label{eq:pb}
	\min_{(x,y) \in \R^{\nx \times \ny}}~~~~ f(x,y),~~~~~~ 
	\mbox{s.t.}~~~~~~ y \in \displaystyle \arg \min_{z \in Z}  g(x,z). 
\end{equation}
wherein we assume that the upper-level function $f(x,y):\R^{\nx \times \ny}\to \R$ is continuous, and $g(x,z):\R^{\nx \times \ny}\to \R$ is such that the lower-level problem $\min_{z \in Z}  g(x,z)$ has a unique solution $y(x)$ for every $x \in \R^{\nx}$, and $Z\subset \R^{\ny}$. Uniqueness of the lower-level problem solution, also known  as the Low-Level Singleton (LLS) 
assumption, is a quite common assumption in many real world applications, such as hyperparameter optimization, meta-learning, pruning, semi-supervised learning on multilayer graphs (see, e.g.,~\cite{franceschi2018bilevel, ji2021bilevel,venturini2023learning, zhang2022advancing}). While for simplicity we focus on the setting described above,
it is important to point out that our analysis still holds, for a specific class of BO problems, even when dropping the LLS assumption (see Remark~\ref{rem_dropLLS}).
 
The algorithms we study here are derivative free optimization (DFO) methods, which do not use derivatives of the upper-level objective function, but rather only the objective value itself. Importantly, in this setting we also assume the availability of some blackbox oracle generating an approximation $\tilde{y}(x)$ of $y(x)$ for any given $x \in \R^\nx$. Among DFO methods, we are interested in particular in direct-search methods (see, e.g.,~\cite{audet2014survey, larson2019derivative}), which sample the objective in suitably chosen tentative points without building a model for it. These algorithmic schemes allow us to prove convergence guarantees under very mild assumptions on our bilevel optimization problem.

\subsection{Previous Work}

Several gradient-based methods have been proposed in the literature to tackle bilevel optimization problems. Those methods usually require the computation of the true objective gradient, called  ``hypergradient'', and rely on the LLS and suitable smoothness assumptions (see, e.g.,~\cite{franceschi2018bilevel,grazzi2020iteration,khanduri2021near,liang2023lower,liu2022bome} and references therein). In another line of research, some asymptotic results based on relaxations of the LLS assumption were also analyzed (see, e.g.,~\cite{liu2021value,liu2021towards,liu2020generic} and references therein). Calculating the hypergradients can be however a notoriously challenging and time consuming task. It indeed requires the handling of $\nabla_x y(x)$,   which in turns involves the calculation of the Hessian matrix related to the $g$ function via the implicit differentiation theorem. In some contexts, 
the hypergradients might not be available at all due to the blackbox nature of the functions describing the problem. These are the reasons why the development of new and efficient zeroth-order/derivative-free approaches is crucial in the BO context.

As for derivative free approaches, classic direct-search (see, e.g.,~\cite{audet2014survey,conn2009introduction,larson2019derivative}) and trust-region methods (see, e.g.,\cite{conn2009introduction,larson2019derivative})  have been applied to BO in~\cite{conn2012bilevel,ehrhardt2021inexact,mersha2011direct, zhang2014bilevel}. In~\cite{mersha2011direct}, a direct-search method for BO assuming the availability of the true objective is described. More specifically, their analysis does not allow for approximation errors in the solution of the lower-level problem, and relies on suitable assumptions making the true objective directionally differentiable. In~\cite{zhang2014bilevel}, the analysis from~\cite{mersha2011direct} is extended considering lower-level inexact solutions with  a stepsize-based adaptive error. In~\cite{conn2012bilevel}, an algorithm applying trust-region methods both in the inner level and on the true objective is described, with an adaptive estimation error for the true objective depending on the trust-region radius; in that work, a strategy to recycle function evaluations for the lower-level problem is described as well. In~\cite{ehrhardt2021inexact}, the analysis of another trust-region method with adaptive error for bilevel optimization is carried out.  The authors report  worst-case complexity estimates  both in terms of upper-level iterations and computational work from the lower-level problem, when considering a strongly convex lower-level problem solved by a suitable gradient descent approach.
In the more recent works~\cite{chen2023bilevel,maheshwari2023convergent}, zeroth-order methods based on smoothing strategies~\cite{nesterov2017random} are analyzed. These studies, drawing inspiration from the complexity results provided in~\cite{jordan2023deterministic} for zeroth-order methods that handle nonsmooth and non-convex objectives, offer complexity estimates tailored for the BO setting. They rely on the assumptions that the lower-level problem can be solved with fixed precision, and that gradient descent on the lower level converges either polinomially or exponentially, respectively. 

Finally, min-max DFO problems (which can be seen as a particular instance of BO) are also recently tackled in the literature~\cite{anagnostidis2021,menickelly2020outer}. Relevant to our work are also direct-search methods under the presence of noise. While previous works analyze direct-search methods with adaptive deterministic~\cite{lucidi2002derivative} and stochastic noise \cite{anagnostidis2021,audet2019stomads,rinaldi2022weak}, we are not aware of previous analyses of direct-search methods with bounded but non adaptive noise. 

\subsection{Contributions}
Our contributions can be summarized as follows.
\begin{itemize}
	\item We define and analyze the first inexact direct-search schemes for BO problems with general potentially nonsmooth true objectives. Those methods indeed never require exact
lower-level problem solutions, but instead assume access
to approximate solutions with fixed accuracy, a reasonable assumption in practice. We therefore operate in a different setting than the one considered in previous works on direct-search for BO, where true objectives are directionally differentiable~\cite{mersha2011direct,zhang2014bilevel} and lower-level solutions are exact~\cite{mersha2011direct}  or require an adaptive precision~\cite{zhang2014bilevel}.
	\item We analyze mesh based direct-search schemes for BO, extending in particular the classic mesh adaptive direct-search (MADS) scheme from~\cite{audet2006mesh}. This is, to the best of our knowledge, the first analysis of this scheme that considers both inexact objective evaluation and the simple decrease condition for new iterates used originally in~\cite{audet2006mesh}.
	\item We give the first convergence results for direct-search schemes with bounded and non-adaptive noise on the objective.
	\item We give the first convergence guarantees to $(\delta, \epsilon)$-Goldstein stationary points for direct-search schemes applied to general nonsmooth objectives. With respect to classic analyses considering Clarke stationary points (see, e.g.,~\cite{audet2017derivative}), these are the first results for direct-search scheme involving some quantitative measure of approximate nonsmooth stationarity. 
\end{itemize}

\section{Background and Preliminaries}
We now introduce the main assumptions considered in the paper, along with a set of helpful preliminary results that will support the subsequent convergence theory. As anticipated in the introduction, we will always assume the existence of a unique minimizer $y(x)$ for the lower-level problem, i.e., that the LLS assumption holds.
\begin{assumption}
	\label{as:pbLower:unique}
	For any $x \in \R^{\nx}$, we have that $\argmin_{z \in Z} g(x, z) = \{y(x)\}$.
\end{assumption}

Under Assumption~\ref{as:pbLower:unique}, the bilevel optimization problem~\eqref{eq:pb} can then be rewritten as 
\begin{equation}\label{eq:pb:reduced}
	\min_{x \in \R^{\nx}}~~~~ F(x):=f\left(x,y(x)\right). 
\end{equation}

However, in practical applications, it is usually necessary to employ an iterative method to compute $y(x)$. Therefore, one cannot expect to obtain an exact value of $y(x)$, but rather some approximation. We  will hance make use of the following assumption.
\begin{assumption}\label{as:prec}
	For all $x \in \R^{\nx}$ we can compute an approximation $\tilde{y}(x)$ of $y(x)$ such that:
	\begin{equation}\label{eq:prec}
		\n{\tilde{y}(x) - y(x)} \leq \varepsilon \, .
	\end{equation}
\end{assumption}
While the remaining assumptions introduced in this section are not always needed, in the rest of this manuscript we always assume that Assumptions~\ref{as:pbLower:unique} and~\ref{as:prec} hold. 
\begin{remark}\label{rem_dropLLS}
	Our analysis extends to the case where  $\argmin_{z \in Z} g(x, z)$ is not a singleton, but  an approximate solution $\tilde{y}(x)$ of the simple bilevel problem
	\begin{equation}\label{eq:pbs}
		\min_{y\in \R^{\ny}}  ~~~~ f(x,y),~~~~~~ 
		\mbox{s.t.}~~~~~~ y \in \displaystyle \arg \min_{z \in Z}  g(x,z). 
	\end{equation}
	is available for every $x \in \R^{\nx}$. In fact our convergence proofs rely on~\eqref{eq:prec} rather than the singleton assumption. We refer the reader to the recent work~\cite{chen2023bilevel} for a detailed discussion on the complexity and regularity properties of the simple bilevel problem~\eqref{eq:pbs}. 
\end{remark}

In the next proposition, we show how condition~\eqref{eq:prec} can be satisfied, by applying gradient descent to $g(x, \cdot)$, under a suitable error bound condition on $\nabla_{y} g(x, y)$ generalizing strong convexity (see, e.g,~\cite{karimi2016linear} for a detailed comparison with other conditions). We also give an explicit bound on the number of iterations needed to satisfy~\eqref{eq:prec}.

\begin{proposition}
	Assume that there exists $c_g>0$ such that for all $y\in Z$ ,
	\begin{equation}\label{eq:gbound}
		c_g\|y - y(x)\| \le \|\nabla_y g(x, y)\|\, .
	\end{equation}
	
	Furthermore, let $\nabla_y g$ be $L_g$ Lipschitz continuous in $y$, uniformly in $x$. Define $y_{0}(x)$ to be any arbitrary initialization mapping onto the domain of $g(x,\cdot)$. Then consider the sequence,
	\begin{equation}
		y_{k + 1}(x) = y_k(x) - \frac{1}{L_g} \nabla_y g(x, y_k(x)) \, .
	\end{equation}
	Define the solution estimate to be:
	\begin{equation}
		\tilde{y}(x) = \argmin_{k \in [0 : K(x)]} \n{\nabla_y g(x, y_k(x))}
	\end{equation}
	It holds that $\tilde{y}(x)$ satisfies~\eqref{eq:prec}, for
	\begin{equation}\label{eq:kxdef}
		K(x) = \left\lceil \frac{2L_g(g(x, y_0(x)) - g(x, y(x)))}{\varepsilon^2 c_g^2} \right \rceil \, .
	\end{equation}
\end{proposition}
\begin{proof}
	This follows from the well known iteration complexity of gradient descent for smooth non convex objectives. 
\end{proof}

We introduce now some technical assumptions on the objective function needed in our analysis. 

\begin{assumption}\label{as:flow}
	The function $f$ is lower bounded by $\flow$.
\end{assumption}

\begin{assumption} \label{as:fLip}
	The function $f$ is Lipschitz continuous with respect to $y$ with Lipschitz constant 
	$L_f$ (independent of $x$).
\end{assumption}
We remark that these assumptions are an adaptation to our bilevel setting of standard assumptions made in the analysis of direct-search methods~\cite{conn2009introduction,lucidi2002derivative}. Assumption~\ref{as:prec} together with Assumption~\ref{as:fLip} imply that $\tilde{F}(x) := f(x, \tilde{y}(x))$ is an approximation of $F(x)$ with accuracy $L_f\varepsilon$. Indeed,
\begin{equation} \label{eq:epsaccuracy}
	| \tilde{F}(x) - F(x) | = |f(x, \tilde y(x)) - f(x, {y}(x))| \leq L_f \|\tilde y(x) - y(x)\| \leq L_f\varepsilon \, .
\end{equation}

Some regularity on the true objective $F(x)$ will always be necessary for our analyses. We consider both the differentiable and the potentially non differentiable setting. 

\begin{assumption} \label{as:Flip}
	$F(x)$ is Lipschitz continuous with constant $L_F$.
\end{assumption}

\begin{assumption} \label{as:nablaFlip}
	The function $F$ is continuously differentiable with 
	Lipschitz continuous gradient, of Lipschitz constant 
	$L$.
\end{assumption}
Note that if $f$ is Lipschitz with respect to $x$, and $y(x)$ Lipschitz continuous with respect to $x$, then Assumption~\ref{as:Flip} is satisfied. Furthermore, in the strongly convex lower-level setting there is an explicit expression for $\nabla F$ (see, e.g.,~\cite[ Equation (3)]{chen2023bilevel}), implying that its Lipschitz continuity follows from that of $y(x)$ together with suitable regularity assumptions on $f$ and $g$. 

\subsection{Algorithm}

In this section, we introduce a general direct-search algorithm for bilevel optimization that embeds both directional direct-search methods with \emph{sufficient decrease} and mesh adaptive direct-search methods with \emph{simple decrease}, as defined in~\cite{conn2009introduction}. 
The methods in the first class sample tentative points along a suitable set of descent directions and then select as the new iterate a point satisfying a sufficient decrease condition. 
The methods in the second class sample the points in a suitably defined mesh, and then select the new iterate according to a simple decrease condition. A tentative point $t$ is hence accepted if the decrease condition 

\begin{equation}\label{dec_cond}
f(t, \tilde{y}(t)) < f(x_k, y_k) - \rho(\alpha_k) 
\end{equation}
is satisfied, for $\rho$ nonnegative function. We have a sufficient decrease when $\rho(t) > 0$ with $\lim_{t \rightarrow 0^+} \rho(t)/t = 0$, and a simple decrease in case $\rho(t) = 0$. These two classes of decrease conditions lead to significant differences in convergence properties and consequently require different choices in the algorithm parameters. They will therefore be analyzed separately in Sections~\ref{s:sdec} and~\ref{s:simpdec} respectively. \\

	\begin{algorithm}
		\caption{DS for bilevel optimization}
		\label{alg:general}
		\begin{algorithmic}[1]  
			\STATE \textbf{Initialization:} Choose $x_0\in\mathbb{R}^{\nx}$, $\alpha_0$ initial stepsize, $\rho: \mathbb{R}_{> 0} \rightarrow \R_{\geq 0}$. Let $y_0 = \tilde{y}(x_0)$ be \\ an approximate minimizer for the lower-level problem in $x_0$. \textbf{Optional:} Let $\Delta_0 = \alpha_0$ \\ be the initial frame size parameter.
			\FOR{$k=0,1,2,\ldots$}
			\STATE  Let $M_k \subset \R^{n_x}$ be a mesh depending on $\alpha_k$ and $x_k$. Let $S_k$ be a finite \\ subset of $M_k$. 
			\IF{$f(t, \tilde{y}(t)) < f(x_k, y_k) - \rho(\alpha_k)$ for some $t \in S_k$}
			\STATE Set $x_{k + 1} = t$, declare the iteration successful, and go to step 13.
			\ENDIF
			\STATE Choose a set of descent directions $D_k$, possibly depending on $\Delta_k$ and such that $\{x_k + \alpha_k d \ | \ d \in D_k\} \subset M_k$. For a given $d \in D_k$, compute the approximate minimizer $y^{\alpha_k d}_{k} = \tilde{y}(x_k + \alpha_k d_k)$ for the lower problem. Evaluate $f$ at the poll points belonging \\ to $\{ (x_k+\alpha_k d, y^{\alpha_k d}_{k}): \, d \in D_k \}$.
			\IF{there exists $d_k \in D_k$ such that $f(x_k+\alpha_k d_k, y^{\alpha_k d_k}_{k}) < f(x_k, y_k) - \rho(\alpha_k)$}
			\STATE Declare the iteration as successful. Set $x_{k+1}=x_k+\alpha_k d_k$ and $y_{k+1}=y^{\alpha_k d_k}_{k}$.
			\ELSE
			\STATE Declare the iteration as unsuccessful. Set $x_{k+1}=x_k$ and $y_{k+1}=y_k$.
			\ENDIF
			\STATE Update the frame size parameter $\Delta_k$ and the stepsize $\alpha_k$.
			\STATE \textbf{Optional:} If some approximate stationarity condition is satisfied, terminate the algorithm.
			\ENDFOR
		\end{algorithmic}
	\end{algorithm}

The detailed scheme (see Algorithm~\ref{alg:general})  follows the lines of the general schemes proposed in~\cite{conn2009introduction} and~\cite{larson2019derivative}, with the addition of calls to the lower-level oracle $\tilde{y}(x)$, and an explicit reference to the mesh used in  mesh-based schemes. At Step 1, the algorithm searches for a new iterate by testing the upper level objective in $(t, \tilde{y}(t))$ for $t$ in $S_k$ subset of the mesh $M_k$. In case Step 1 is not successful, the method generates,  at Step~2, a new iterate by selecting a set of descent directions $D_k$ and testing the upper level objective in $(t, \tilde{y}(t))$ for $t$ chosen along the descent directions using a stepsize $\alpha_k$. Step 3 and Step 4 perform updates on the algorithm iterate and parameters based on the outcome of Step 1 and 2. 
For the set of directions $D_k$, we require in some cases a positive cosine measure, that is
\begin{equation} \label{eq:cmDk}
	\cm(D_k) \stackrel{d}{=} \min_{v \neq 0_{\real^{\nx}}} 
	\max_{d\in D_k} \frac{d^\top v}{\|d\|\|v\|} 
	\ge \kappa \, ,
\end{equation}
for some $\kappa > 0$. 

\section{Sufficient decrease condition} \label{s:sdec}
 In this section, we analyze directional direct-search methods using a sufficient decrease condition with $\rho(t) = \frac{c}{2}t^2$. We first focus on potentially nonsmooth objectives, and then on smooth ones. In both cases we consider the scheme presented in Algorithm~\ref{alg:2}, which can be viewed as an adaption to BO of classic generating set of search directions (GSS) schemes (see, e.g.,~\cite[Algorithm 3.2]{kolda2003optimization}). In order to handle the error introduced by the approximate solution in the lower level, we lower bound the stepsize with a constant $\Alpha$. We further notice that, thanks to the sufficient decrease condition, maintaining a mesh is not necessary, and therefore we simply set $M_k = \R^{\nx}$.

\begin{algorithm}
	\caption{Inexact directional DS for bilevel optimization}
	\label{alg:2}
	\begin{algorithmic}[1]
		\STATE \textbf{Initialization:} Choose starting point $x_0\in\mathbb{R}^{\nx}$, stepsize lower bound $\Alpha \geq 0$, initial stepsize $\alpha_0\geq \Alpha$, coefficient for stepsize contraction $0 < \theta < 1$, coefficient for stepsize expansion $\gamma \geq 1$, sufficient decrease condition coefficient $c$. Let $y_0 = \tilde{y}(x_0)$ be an approximate minimizer for the lower-level problem at $x_0$.    
		
		\FOR{$k=0,1,2,\ldots$}
		\STATE Let $S_k \subset \R^{n_x}$ with $|S_k| < +\infty$. 
		\IF{$f(t, \tilde{y}(t)) < f(x_k, y_k) - \frac{c}{2} \alpha_k^2$ for some $t \in S_k$}
		\STATE Set $x_{k + 1} = t$, declare the iteration successful, and go to step 13.
		\ENDIF
		\STATE Choose a set of descent directions $D_k$. For a given $d \in D_k$, compute the approximate minimizer $y^{\alpha_k d}_{k} = \tilde{y}(x_k + \alpha_k d_k)$ for the lower problem. Evaluate $f$ at the poll points belonging to $\{ (x_k+\alpha_k d, y^{\alpha_k d}_{k}): \, d \in D_k \}$.
		\IF{for some $d_k \in D_k$, $f(x_k+\alpha_k d_k, y^{\alpha_k d_k}_{k}) < f(x_k, y_k) - \frac{c}{2}\alpha_k^2$}
		\STATE Declare the iteration as successful. Set $x_{k+1}=x_k+\alpha_k d_k$ for $d_k$ satisfying the condition and $y_{k+1}=y^{\alpha_k d_k}_{k}$.
		\ELSE
		\STATE Declare the iteration as unsuccessful. Set $x_{k+1}=x_k$ and $y_{k+1}=y_k$.
		\ENDIF
		\STATE \textbf{If} the iteration was successful \textbf{then} maintain or increase
		the corresponding stepsize parameter -- set
		$\alpha_{k+1}$ $\in [\alpha_k, \gamma \alpha_k]$. \textbf{Else} decrease the stepsize parameter, by choosing $\alpha_{k+1}= \max \{ \Alpha,\theta \alpha_k \}$.
		\STATE \textbf{[Optional]} If some approximate stationarity condition is satisfied, terminate the algorithm.
		\ENDFOR
	\end{algorithmic}
\end{algorithm}

\subsection{Nonsmooth objectives}\label{s:nsosd}

First, we present convergence guarantees and proofs thereof for a variant of Algorithm~\ref{alg:2} designed for the case of Lipschitz continuous true objectives, i.e., under Assumption~\ref{as:Flip}. With respect to the general scheme presented as Algorithm~\ref{alg:2}, here $D_k = \{g_k\}$ with $g_k$ generated in the unit sphere. We remark that this is a standard choice for direct-search algorithms applied to nonsmooth objectives (see, e.g.,~\cite[Algorithm $\text{DFN}_{simple}$]{fasano2014linesearch}). The stepsize lower bound here must be strictly positive (i.e. $\Alpha > 0$).This together with the sufficient decrease conditions ensures that the sequence generated by the algorithm is eventually constant, as proved in Lemma~\ref{l:ec}. We then use a novel argument to prove that the limit point of the sequence is a $(\delta, \epsilon)$-Goldstein stationary point. Although such a notion of stationarity has recently gained attention in the analysis of zeroth-order smoothing-based approaches~\cite{jordan2023deterministic,lin2022gradient,rando2023optimal}, including extensions to BO~\cite{chen2023bilevel,maheshwari2023convergent}, to the best of our knowledge, it has never been used for the analysis of direct-search methods.It is further important to notice that convergence of directional direct-search methods to $(\delta, \epsilon)-$Goldstein stationary points in the nonsmooth case is a novel result also for classic optimization problems.
We now recall some useful definitions. If $B_{\delta}(x)$ is the ball of radius $\delta$ centered in $x$, then the $\delta$-Goldstein subdifferential (see, e.g.,~\cite{lin2022gradient}) is defined as
\begin{equation}
	\partial_{\delta} F(x)	= \conv \left\{  \bigcup_{y \in B_{\delta}(x)} \partial F(y) \right\} \, ,
\end{equation}
and $x$ is an $(\delta, \epsilon)$-Goldstein stationary point for the function $F$ if, for some $g \in \partial_{\delta}F(x)$, we have $\n{g}\leq \epsilon$. 

We can now proceed with our convergence analysis. As anticipated, we start by proving that the sequence of iterates generated by our method is eventually constant.
 \begin{lemma} \label{l:ec}
	Let Assumptions~\ref{as:flow} and~\ref{as:fLip} hold. Then there exists $\bar{k}\in \mathbb{N}_0$ such that the sequence $\{x_k\}$ generated by Algorithm~\ref{alg:2} is constant for $k \geq \bar{k}$.
\end{lemma}
\begin{proof}
	Notice that $\{\tilde{F}(x_k)\}$ is non-increasing, with $\tilde{F}(x_k) = \tilde{F}(x_{k + 1})$ after an unsuccessful step, and 
	\begin{equation} \label{eq:succ}
		\tilde{F}(x_{k + 1}) < \tilde{F}(x_k) - \frac{c}{2}\alpha_k^2 \leq \tilde{F}(x_k) - \frac{c}{2} \Alpha^2 
	\end{equation}
	after a successful step. Thus there can be at most
	\begin{equation} \label{eq:numsucc}
		\frac{2\left(\tilde{F}(x_0) - \inf_{x \in \R^n} \tilde{F}(x)\right)}{c \Alpha^2} \leq \frac{2\left(\tilde{F}(x_0) - f_{\textnormal{low}} + L_f\varepsilon\right)}{c \Alpha^2} \, 
	\end{equation}
	successful steps, where we used $\tilde{F}(x) \geq F(x) - L_f\varepsilon \geq f_{\textnormal{low}} - L_f \varepsilon$ in the inequality. Since this quantity is finite, this implies that $\{x_k\}$ is eventually constant. 
\end{proof}
  
We now prove convergence of our algorithm to $(\delta,\epsilon)$-Goldstein stationary points. In order to get our convergence result, 
	we need to assume that the sequence $\{g_k\}$ is dense in the unit sphere. We remark that such a dense sequence can be generated using a suitable quasirandom sequence (see, e.g.,~\cite{halton1960efficiency, liuzzi2019trust}).
\begin{theorem} \label{th:nonsmooth}
	Let Assumptions~\ref{as:flow},~\ref{as:fLip} and~\ref{as:Flip} hold. Assume that $\{g_k\}$ is dense in the unite sphere. Then the sequence $\{x_k\}$ generated by Algorithm~\ref{alg:2} is eventually constant, with the unique limit point $(\delta,\epsilon)$-Goldstein stationary, for
	\begin{equation}
		\epsilon = \frac{4L_f\varepsilon}{\Alpha} + c\Alpha ~~\mbox{and}  \quad \delta = \Alpha \, . 
	\end{equation}
\end{theorem}
\begin{proof}
	First, $\{x_k\}$ is eventually constant as seen in Lemma~\ref{l:ec}. Let $\bar{x}$ be the unique limit point. By the stepsize updating rule, we have that every iteration must be unsuccessful with $\alpha_k = \Alpha$ for $k$ large enough. Then, there exists $\bar{k} \in \mathbb{N}$ large enough such that for every $k \geq \bar{k}$

	\begin{equation}
		\tilde{F}(\bar{x}) < \tilde{F}(\bar{x} + \alpha_k g_k) + \frac{c}{2}\Alpha^2 = \tilde{F}(\bar{x} + \Alpha g_k) + \frac{c}{2}\Alpha^2 
	\end{equation}
	implying
	\begin{equation}
		F(\bar{x}) < F(\bar{x} + \Alpha g_k) + \frac{c}{2}\Alpha^2 + 2L_f\varepsilon\, .
	\end{equation}
	By the density of $\{g_k\}$ it follows
	\begin{equation}\label{eq:Fnonsmooth}
		F(\bar{x}) < F(\bar{x} + d) + \frac{c}{2}\Alpha^2 + 2L_f\varepsilon
	\end{equation}
	for every $d$ such that $\n{d} = \Alpha$. \\
	We now define the function $\bar{F}_{\bar{x}}(d) := F(\bar{x} + d) + (\frac{c}{2} + \frac{2L_f\varepsilon}{\Alpha^2}) \n{d}^2$. Since 
	\begin{equation}
		\bar{F}_{\bar{x}}(0) < \bar{F}_{\bar{x}}(d)
	\end{equation}
	for every $d$ such that $\n{d} = \Alpha$ by~\eqref{eq:Fnonsmooth}, there must be a $\tilde{d} \in \argmin_{\n{d} \leq \Alpha} \bar{F}_{\bar{x}}(d)$ with $\n{\tilde{d}} < \Alpha$. We can conclude
	\begin{equation}
		0 \in \partial \bar{F}_{\bar{x}}(\tilde{d}) = \partial F(x + \tilde{d}) - \left(c + \frac{4L_f\varepsilon}{\Alpha^2}\right)\tilde{d}
	\end{equation} 
	Equivalently, $g = (c + \frac{4L_f\varepsilon}{\Alpha^2})\tilde{d} \in  \partial F(x + \tilde{d})$ and since $\partial F(x + \tilde{d}) \subset \partial_{\Alpha} F(\bar{x})$ we have $g \in \partial_{\Alpha} F(\bar{x})$. To conclude, observe $\n{g} < c \Alpha + \frac{4L_f\varepsilon}{\Alpha} $. 
\end{proof}

As a corollary of Theorem~\ref{th:nonsmooth}, for $\Alpha \propto \sqrt{\varepsilon}$ we are able to get a $(\mathcal{O}(\sqrt{\varepsilon}), \mathcal{O}(\sqrt{\varepsilon}))-$Goldstein stationary point. Interestingly, the order of magnitude $\mathcal{O}(\sqrt{\varepsilon})$ of the approximation error coincides with that of typical gradient approximation methods~\cite{berahas2022theoretical}, as well as with that of direct-search in the smooth setting, as we shall see in the next section. 
\begin{corollary}
	Let Assumptions~\ref{as:flow},~\ref{as:fLip} and~\ref{as:Flip} hold. Assume that $\{g_k\}$ is dense in the unite sphere. Then the sequence $\{x_k\}$ generated by Algorithm~\ref{alg:2} with $\Alpha = 2\sqrt{\frac{L_f\varepsilon}{c}}$ is eventually constant, with the unique limit point $(\delta,\epsilon)$-Goldstein stationary, for
	\begin{equation}
		\epsilon = 4\sqrt{L_f\varepsilon c}~~\mbox{and}  \quad \delta = 2\sqrt{\frac{L_f\varepsilon}{c}} \, .
	\end{equation}
\end{corollary}

\subsection{Smooth objectives}
We now focus on the case where the objective $F$ is smooth, in particular under Assumption~\ref{as:nablaFlip}. We consider here a variant of Algorithm~\ref{alg:2} with $D_k$ positive spanning set. When the stepsize lower bound is strictly positive we set as termination criterion $\alpha_k = \alpha_{k + 1} = \Alpha$. Our scheme can hence be seen as a variant of classic direct-search methods for smooth objectives~\cite{conn2009introduction,kolda2003optimization}.
It is important to highlight that this is the first analysis of direct-search methods for smooth objectives under bounded noise. The only analysis of direct-search methods we are aware of in the smooth case is the one given in~\cite{dzahini2022expected} under stochastic noise, where, however, the author only focuses on classic optimization problems. 

We first extend to our bounded error setting a standard result that allows to get an upper bound on the gradient norm for unsuccessful iterations (see, e.g.,~\cite[Theorem 3.3]{kolda2003optimization}). 

\begin{lemma} \label{l:unsucc}
	Let Assumptions~\ref{as:fLip} and~\ref{as:nablaFlip} hold, together with~\eqref{eq:cmDk}. Let $\{x_k\}$ be a sequence generated by Algorithm~\ref{alg:2}. If the iteration $k$ is unsuccessful, then
	\begin{equation}\label{eq:ngfkbound}
		\n{\nabla F(x_k)} \leq \frac{1}{\kappa} \left(\frac{(L + c)\alpha_k}{2} + \frac{2L_f\varepsilon}{\alpha_k} \right) \, .
	\end{equation}
\end{lemma}

\begin{proof}
	Let $d \in D_k$ be such that
	\begin{equation} \label{eq:cmdd}
		-\nabla F(x_k)^\top d \geq \kappa \n{\nabla F(x_k)} \n{d} \, . 
	\end{equation}
	We have
	\begin{equation} \label{eq:lboundnabla}
		\begin{aligned}
			& \kappa \alpha_k \n{\nabla F(x_k)} \n{d}  - \alpha_k^2 \frac{L}{2}\n{d}^2 \leq - \alpha_k \nabla F(x_k)^\top d - \alpha_k^2 \frac{L}{2}\n{d}^2 \\ 
			& \leq F(x_k) - F(x_k + \alpha_k d) 	\leq \tilde{F}(x_k) - \tilde{F}(x_k + \alpha_k d) + 2L_f\varepsilon \leq \frac{c}{2}\alpha_k^2 + 2L_f\varepsilon \, ,
		\end{aligned}
	\end{equation}
	where we used~\eqref{eq:cmdd} in the first inequality, the standard descent lemma in the second inequality,~\eqref{eq:epsaccuracy} in the third inequality, and that the step is unsuccessful in the last inequality. Therefore, since by assumption $\n{d} = 1$
	\begin{equation}
		\kappa \alpha_k \n{\nabla F(x_k)} =	\kappa \alpha_k \n{\nabla F(x_k)} \n{d}  \leq \frac{c}{2}\alpha_k^2 + 2L_f\varepsilon \alpha_k^2 \frac{L}{2}\n{d}^2 = \frac{c}{2}\alpha_k^2 + 2L_f\varepsilon+ \alpha_k^2 \frac{L}{2} \, ,
	\end{equation}
	implying the thesis.
\end{proof}

We now prove convergence and complexity bounds when $\Alpha > 0$, extending those given in~\cite{vicente2013worst} for the exact oracle case, and $\Alpha=0$. We notice that in this second case we lose finite convergence and our  guarantees are thus somewhat weaker, i.e., we are only able to prove that the stepsize converges to $0$ and that at some point the gradient norm is $\mathcal{O}(\sqrt{\varepsilon})$. 
\begin{theorem} \label{th:smoothcompl}
Let Assumptions~\ref{as:flow},~\ref{as:fLip} and~\ref{as:nablaFlip} hold, together with~\eqref{eq:cmDk} for every $k \in \mathbb{N}_0$. Let $\{x_k\}$ be a sequence generated by Algorithm~\ref{alg:2}.
	\begin{enumerate}
		\item 	If $\Alpha > 0$, then the algorithm terminates after $\bar{k}$ iterations, with
		\begin{equation}
			\bar{k} < 1 + \frac{2}{\Alpha^2c}(\tilde F(x_0) - f_{low} + 2L_f\varepsilon)\left(1 - \frac{\ln \gamma}{\ln \theta}\right) + \frac{\ln \Alpha - \ln \alpha_0}{\ln \theta} \, ,
		\end{equation}
		and its last iterate $x_{\bar{k}}$ is such that
		\begin{equation}
			\n{\nabla F(x_{\bar{k}})} \leq \frac{1}{\kappa} \left( \frac{(L + c)\Alpha}{2} + \frac{2L_f\varepsilon}{\Alpha} \right) \, .
		\end{equation}
		\item If, furthermore, it holds that $\Alpha = 2\sqrt{\frac{L_f \varepsilon}{L + c}}$, then
		\begin{equation} \label{eq:glowbound}
			\n{\nabla F(x_{\bar{k}})} \leq \frac{2}{\kappa} \sqrt{(c + L) L_f \varepsilon} \, .
		\end{equation}
		\item If $\Alpha = 0$, then $\alpha_k \rightarrow 0$, and if additionally $\alpha_0 \geq \barAlpha= 2\sqrt{\frac{L_f \varepsilon}{L + c}}$, for some $\bar{k} \in \mathbb{N}_0$ we have
		\begin{equation} \label{eq:gbound2}
			\n{\nabla F(x_{\bar{k}})} \leq \frac{1}{\theta\kappa} \left( \frac{(L + c)\barAlpha}{2} + \frac{2L_f\varepsilon}{\barAlpha}\right) \, ,
		\end{equation}
		and 
		\begin{equation} \label{eq:Fbound}
			F(x_k) \leq F(x_{\bar{k}}) + 2L_f\varepsilon \quad \textnormal{ for all }k\geq \bar{k} \, .
		\end{equation}
	\end{enumerate}
\end{theorem}
\begin{proof}
	1.	Let $k_s$ and $k_{ns}$ be the number of successful and unsuccessful steps, so that $k_s + k_{ns} = k$. Reasoning as in Lemma~\ref{l:ec}, we obtain by~\eqref{eq:numsucc}
	\begin{equation} \label{eq:ksub}
		k_s < \frac{2}{\Alpha^2c}(F(x_0) - f_{\textnormal{low}} + 2L_f\varepsilon) \, .
	\end{equation} 
	Furthermore, since
	\begin{equation}
		\Alpha \leq \alpha_k \leq \alpha_0\gamma^{k_{s}}\theta^{k_{ns} - 1} \, , 
	\end{equation}
	we get
	\begin{equation} \label{eq:knssub}
		\begin{aligned}
			& k_{ns} \leq 1 -\frac{1}{\ln(\theta)}(\ln(\alpha_0) - \ln(\Alpha) + k_{s}\ln(\gamma)) \\ 
			& \leq 1 -\frac{1}{\ln(\theta)}(\ln(\alpha_0) - \ln(\Alpha) + \frac{2}{\Alpha^2c}(\tilde F(x_0) - f_{\textnormal{low}} + 2L_f\varepsilon)\ln(\gamma))	\, ,	
		\end{aligned}
	\end{equation}
	where we applied~\eqref{eq:ksub} in the second inequality. 
	Combining the bounds on the successful and unsuccessful steps~\eqref{eq:ksub} and~\eqref{eq:knssub}, we have
	\begin{equation}
		k = k_{s} + k_{ns}	< 1 + \frac{2}{\Alpha^2c}(\tilde F(x_0) - f_{low} + 2L_f\varepsilon)\left(1 - \frac{\ln \gamma}{\ln \theta}\right) + \frac{\ln \Alpha - \ln \alpha_0}{\ln \theta} \, ,
	\end{equation}
	as desired. \\
	2. Follows from a direct application of the first result. \\
	3. Reasoning as in the first result, the number of successful steps with stepsize above a certain threshold is bounded, hence $\alpha_k \rightarrow 0$. Furthermore, for any $\bar{k} \in \mathbb{N}_0$, if $k \geq \bar{k}$
	\begin{equation}
		F(x_k) \leq \tilde{F}(x_{k}) + L_f\varepsilon \leq \tilde{F}(x_{\bar{k}}) + L_f\varepsilon \leq F(x_{\bar{k}}) + 2L_f\varepsilon \, ,
	\end{equation} 
	which proves~\eqref{eq:Fbound}. Let $\barAlpha = 2\sqrt{\frac{L_f \varepsilon}{L + c}}$. 
	Since $\alpha_0 \geq \barAlpha$, and $\alpha_k \rightarrow 0$ with contraction factor $\theta$, we must have 
	$\alpha_{\bar{k}} \in [\theta \barAlpha, \barAlpha]$ for some $\bar{k} \in \mathbb{N}_0$. 
	Then~\eqref{eq:gbound2} follows from~\eqref{eq:ngfkbound} for $\alpha_k = \alpha_{\bar{k}}$.
\end{proof}

We now extend to our setting the $\mathcal{O}(n^2/\epsilon^2)$ complexity result given in~\cite[Corollary 2]{vicente2013worst}. For a fixed precision $\epsilon$, an approximation error $\varepsilon = \mathcal{O}(\epsilon^2)$ is required, as for classic gradient approximation schemes~\cite{berahas2022theoretical}.
\begin{corollary}
	Let Assumptions~\ref{as:flow},~\ref{as:fLip} and~\ref{as:nablaFlip} hold, together with~\eqref{eq:cmDk} for every $k \in \mathbb{N}_0$. Let $\{x_k\}$ be a sequence generated by Algorithm~\ref{alg:2}. Assume also $\varepsilon \leq \epsilon^2 \kappa^2$, that at every iterations there are at most $d_1n$ function evaluations and that $\kappa \geq d_2/\sqrt{n}$, for $d_1, d_2 > 0$. Then if $\Alpha~=~2\sqrt{\frac{L_f \varepsilon}{L + c}}$, the algorithm terminates after $\mathcal{O}(n^2/\epsilon^2)$ function evaluations with $\n{\nabla f(x_{\bar{k}})} \leq d_3 \epsilon$, for $d_3 > 0$ depending only on $c, L$ and $L_f$.
\end{corollary}
\begin{proof}
	Follows from point 1 and 2 of Theorem~\ref{th:smoothcompl}, plugging in the parameters specified in the assumptions.
\end{proof}
\section{Simple decrease condition} \label{s:simpdec}
In this section, we analyze two methods based on simple decrease condition (i.e., with $\rho(t)~=~0$, in~\eqref{dec_cond}), one for potentially nonsmooth objectives and one for smooth objectives. Both methods follow the scheme presented in Algorithm~\ref{alg:3}, which is an adaptation to the BO setting of the mesh adaptive direct-search algorithm (MADS, see~\cite{audet2014survey} and references therein). Again we lower bound the stepsize by a constant $\Alpha$. The stepsize updating rule we use to handle unsuccessful iterations depends on the mesh size parameter $\Delta_k$ and the contraction coefficient $\theta$, and smoothness of the true objective (i.e., update varies between the smooth and the nonsmooth case).

It is a standard assumption in the analysis of MADS that all the iterates lie in a compact set (see, e.g.,~\cite[Section 3]{audet2006mesh}). In our framework, this can be ensured  if the following boundedness assumption is satisfied. 
\begin{assumption} \label{as:comp}
	The set 
	\begin{equation}
		\mathcal{L}_{\varepsilon} = \{x \in \R^{n_x} \ | \ F(x) \leq F(x_0) + 2L_f\varepsilon \}
	\end{equation}
is bounded.
\end{assumption}

The mesh, as defined in the literature (see,e.g.,~\cite{audet2017derivative,conn2009introduction} and references therein for further details), is a discrete set of points from which the  algorithm selects candidate trial points. Its coarseness is parameterised by the mesh size parameter $\delta$. The goal of each
algorithm iteration is to get a mesh point whose objective function value improves with respect to the incumbent value.
Given a positive spanning set $D$ and a center $x$ the related mesh is formally defined as follows:
\begin{equation} \label{eq:mesh}
	M = \{ x + \delta Dy \ | \ y \in \mathbb{N}^p\} \, ,
\end{equation}
where, with a slight abuse of notation, we use $D$ also for the matrix $D\in\R^{n \times p}$ with columns corresponding to the elements of the set $D$. We notice that the mesh is just a conceptual tool, and is never actually constructed. 

\begin{algorithm}
	\caption{Inexact mesh based DS for bilevel optimization}
	\label{alg:3}
	\begin{algorithmic}[1]
		\STATE \textbf{Initialization:} Choose starting point $x_0\in\mathbb{R}^{\nx}$, stepsize lower bound $\Alpha \geq 0$, initial mesh size parameter $\alpha_0 = \Alpha \theta^{-\mu_0}$, with $\mu_0 \in \mathbb{N}_0$, starting frame parameter $\Delta_0 = \alpha_0$, stepsize contraction/expansion parameter $\theta \in (0, 1) \cap \mathbb{Q}$, $G \in \R^{n \times n}$ invertible and $Z \in \mathbb{Z}^{n \times p}$ with columns forming a positive spanning set. Let $D = GZ$. Let $y_0 = \tilde{y}(x_0)$ \\ be an approximate minimizer for the lower-level problem in $x_0$.    
		
		\FOR{$k=0,1,2,\ldots$}
		\STATE \textbf{[Optional]} Let $M_k$ be the mesh with size parameter $\alpha_k$, positive spanning set $D$ \\ and center $x_k$. Select a finite subset $S_k$ of $M_k$. 
		\IF{$f(t, \tilde{y}(t)) < f(x_k, y_k)$ for some $t \in S_k$}
		\STATE Set $x_{k + 1} = t$, declare the iteration successful, and \textbf{go to step 7}.
		\ENDIF
		\STATE Choose a positive spanning set $D_k$ such that $\{x_k + \alpha_kd \ | \ d \in D_k\} \subset M_k$. Compute \\ the approximate minimizer $y^{\alpha_k d}_{k} = \tilde{y}(x_k + \alpha_k d_k)$ for the lower problem. Evaluate $f$ \\ at the poll points belonging to $\{ (x_k+\alpha_k d, y^{\alpha_k d}_{k}): \, d \in D_k \}$.
		\IF{there exists $d_k \in D_k$ such that $f(x_k+\alpha_k d_k, y^{\alpha_k d_k}_{k}) < f(x_k, y_k)$}
		\STATE Declare the iteration as successful. Set $x_{k+1}=x_k+\alpha_k d_k$ and $y_{k+1}=y^{\alpha_k d_k}_{k}$.
		\ELSE
		\STATE Declare the iteration as unsuccessful. Set $x_{k+1}=x_k$ and $y_{k+1}=y_k$.
		\ENDIF
		\STATE \textbf{If} the iteration was successful \textbf{then} set $\Delta_{k + 1} = \theta^{-1} \Delta_k$ and  $\alpha_k = \min(\Delta_k, \Delta_k^2)$. \\ \textbf{Else}  set $\Delta_{k+1}= \max \{ \Alpha, \theta \Delta_k \}$ and $\alpha_{k + 1} = \alpha_u(\alpha_k, \Delta_k, \theta)$.
		\STATE \textbf{[Optional]} If some approximate stationarity condition is satisfied, terminate the algorithm.
		\ENDFOR
	\end{algorithmic}
\end{algorithm}

\subsection{Nonsmooth objectives}
With respect to the general scheme presented in Algorithm~\ref{alg:3}, here the stepsize updating rule for unsuccessful iterations is given by $\alpha_u(\alpha_k, \Delta_k, \theta) = \min(\Delta_k, \Delta_k^2, \theta\alpha_k)$, ensuring that $\alpha_k \rightarrow 0$ and the mesh gets infinitely dense if the algorithm gets stuck in a certain point. The set of search directions $D_k$ must be such that
\begin{equation} \label{eq:dbracket2}
	\frac{\Delta_k}{\alpha_k} b_1(\alpha_k) \leq \n{d} \leq \frac{\Delta_k}{\alpha_k} b_2(\alpha_k)
\end{equation}
for all $d \in D_k$, with $b_i: \R_{> 0} \rightarrow \R_{> 0}$ such that $\lim_{t \rightarrow 0}b_i(t) = 1$ for $i\in\{1, 2\}$. Thus with respect to the classic MADS scheme here the frame size $\Delta_k$ defines also a lower bound and not only an upper bound on the distance between the current iterate and tentative points selected in the poll step. This adjustment is necessary due to the error on the true objective evaluation. As shown in the next lemma, Condition~\eqref{eq:dbracket2} ensures that as the stepsize converges to $0$ the tentative steps get closer and closer the boundary of a ball of radius $\Alpha$.
\begin{lemma} \label{l:spherelimit}
	Assume that $\Alpha > 0$ and that~\eqref{eq:dbracket2} holds. Then if $\lim_{k \in K} \alpha_k = 0$, the set of limit points of $\{\alpha_kD_k\}_{k \in K}$ is contained in $S^{n_x - 1}(\Alpha)$.
\end{lemma}
\begin{proof}
	If $\lim_{k \in K} \alpha_k = 0$ then it holds that, for $k \in K$ large enough, $\Delta_k = \Alpha$. Consider $\{d_k\}= D_k$. It holds that, for all $d_k$,
	\begin{equation}
		\limsup_{k \in K} \n{\alpha_k d_k} \leq \limsup_{k \in K} \Delta_k b_2(\alpha_k) = \Alpha \, ,
	\end{equation}
	where we applied~\eqref{eq:dbracket2} in the inequality.
	Analogously, we can prove $\liminf_{k \in K} \n{\alpha_k d_k} \geq \Delta_k$, whence $\lim_{k \in K} \n{\alpha_k d_k} = \Alpha$, which implies the thesis. 
\end{proof}

We now extend to this scheme the $(\delta, \epsilon)$-Goldstein stationarity result proved under the sufficient decrease condition in Section~\ref{s:nsosd}. Also in this case we are not aware of any analogous result for the standard MADS scheme, which is instead known to convergence to Clarke stationary points~\cite{audet2006mesh}. \\
We start with a lemma that extends  a well known property of   MADS (see, e.g.,~\cite[Proposition~3.1]{audet2006mesh}) to our bilevel setting.
\begin{lemma}
	Let Assumptions~\ref{as:fLip},~\ref{as:Flip} and~\ref{as:comp} hold. Then the sequence $\{\alpha_k\}$ generated by Algorithm~\ref{alg:MADS} is such that $\liminf \alpha_k = 0$. 
\end{lemma}
\begin{proof}
	Since $\{\tilde{F}(x_k)\}$ is  non-increasing (and strictly decreasing for successful iterations), $\{x_k\}$ is contained in the set $\mathcal{L}_{\varepsilon}$, which is compact by Assumptions~\ref{as:Flip} and~\ref{as:comp}. Thus $\liminf \alpha_k = 0$ follows from the finiteness of feasible points generated in $\mathcal{L}_{\varepsilon}$ when keeping the parameter $\alpha_k$ lower bounded, which can be proved with the same arguments used for MADS in~\cite[Proposition 3.1]{audet2006mesh}.
\end{proof}

We can now state our main result.

\begin{theorem}\label{th:nssimd}
	Let Assumptions~\ref{as:fLip},~\ref{as:Flip} and~\ref{as:comp} hold. Let $K$ be a subset of  unsuccessful iteration indices  related to Algorithm~\ref{alg:MADS}. Let us further assume that:
	\begin{itemize}
        \item $\lim_{k \in K} x_k = \bar{x}$;
		\item $\lim_{k \in K} \alpha_k = 0$;
		\item $\{\hat{D}_k\}_{k \in K}$ is dense in the unit sphere, with $\hat{D}_k = \{ \frac{d}{\n{d}} \ | \ d\in D_k\}$;
		\item Condition~\eqref{eq:dbracket2} holds.
	\end{itemize}
Then, the limit point $\bar{x}$ of $\{x_k\}_{k \in K}$ is $(\delta, \epsilon)$-Goldstein stationary, for
	\begin{equation}
		\epsilon = \frac{4L_f\varepsilon}{\Alpha}~~\mbox{and}  \quad  \delta= \Alpha.
	\end{equation}
\end{theorem}
\begin{proof}
	Let $\bar{d} \in \R^n$ with $\n{\bar{d}} = 1$, and let $L \subset K$ be such that $\lim_{k \in L} \frac{d_k}{\n{d_k}} \rightarrow \bar{d}$, with $d_k \in D_k$. Then $\alpha_k d_k \rightarrow \Alpha \bar{d}$ by Lemma~\ref{l:spherelimit}. Now, for every $k \in L$
	\begin{equation} \label{eq:madsdeltaf}
		F(x_k) - F(x_k + \alpha_k d_k) \leq \tilde{F}(x_k) - \tilde{F}(x_k + \alpha_k d_k) + 2L_f\varepsilon \leq 2L_f\varepsilon\, , 
	\end{equation}
	where the first inequality follows from~\eqref{eq:epsaccuracy}, and we used that the step $k$ is unsuccessful in the second inequality. Passing to the limit, we obtain 
	\begin{equation} \label{eq:madsdineq}
		F(\bar{x}) \leq F(\bar{x} + \Alpha\bar{d}) + 2L_f\varepsilon \, .
	\end{equation} 
	Now let $\bar{F}_{\bar{x}}(d) = F(\bar{x} + d) + \frac{2L_f\varepsilon}{\Alpha^2} \n{d}^2$.
	By applying~\eqref{eq:madsdeltaf} we get 
	$$\bar{F}_{\bar{x}}(0) \leq \bar{F}_{\bar{x}}(\Alpha\bar{d})\, , $$
	and given that $\bar{d}$ is arbitrary, this holds for any $d$ such that $\n{d} = \Alpha$. The thesis then follows as in the proof of Theorem~\ref{th:nonsmooth}.
\end{proof}

As in Section~\ref{s:nsosd}, here we also have a corollary showing that for $\Alpha \propto \sqrt{\varepsilon}$ we are able to get a $(\mathcal{O}(\sqrt{\varepsilon}), \mathcal{O}(\sqrt{\varepsilon}))$-Goldstein stationary point. 
\begin{corollary}
	Under the assumptions of Theorem~\ref{th:nssimd},
the limit point $\bar{x}$ of  the sequence $\{x_k\}$ generated by Algorithm~\ref{alg:suff_dec_ns} with $\Alpha = 2\sqrt{L_f\varepsilon}$  is
 $(\delta,\epsilon)$-Goldstein stationary, for
	\begin{equation}
		\epsilon = \delta = 2\sqrt{L_f\varepsilon} \, .
	\end{equation}
\end{corollary}
\subsection{Smooth objectives}
Now we consider the case where the true objective is smooth, i.e., Assumption~\ref{as:nablaFlip} holds. With respect to the general scheme reported in Algorithm~\ref{alg:3}, we have $\alpha_u(\alpha_k, \Delta_k, \theta) = \min(\Delta_k, \Delta_k^2)$, and the algorithm terminates if $\alpha_k = \alpha_{k + 1} = \Alpha$. As for $D_k$, it must always satisfy $\cm(D_k) \geq \kappa$ for some positive $\kappa$ independent from $k$, as well as
\begin{equation} \label{eq:dbrackets}
	\frac{\Delta_k}{\alpha_k}b_1 \leq \n{d} \leq \frac{\Delta_k}{\alpha_k}b_2 
\end{equation}
for every $d \in D_k$. \\
We remark that convergence of mesh based schemes for smooth objectives is well understood (see, e.g.,~\cite[Chapter 7]{audet2017derivative}), so that once again our main contribution here is the adaptation to the bilevel setting. We begin our analysis by extending Lemma~\ref{l:unsucc} under the simple decrease condition and condition~\eqref{eq:dbrackets} on the descent directions.
\begin{lemma} \label{l:unsucc2}
	Let Assumptions~\ref{as:fLip} and~\ref{as:nablaFlip} hold, together with~\eqref{eq:cmDk}. Let $\{x_k\}$ be a sequence generated by Algorithm~\ref{alg:MADS}.
  If the step $k$ is unsuccessful, then
	\begin{equation}\label{eq:ngfkbound2}
		 \n{\nabla F(x_{k})}  \leq \frac{1}{\kappa} \left( \frac{b_2 \Delta_k L}{2} + \frac{2L_f\varepsilon}{b_1 \Delta_k} \right) \, .
	\end{equation}
\end{lemma}
\begin{proof}
		Since the step is unsuccessful, by
	considering $d \in D_{k}$ such that
	\begin{equation}
		-\nabla F(x_k)^{\top} d \geq \kappa \n{\nabla F(x_k)} \n{d}
	\end{equation}	
	we have, reasoning as in~\eqref{eq:lboundnabla} with $c = 0$
	\begin{equation}
		\kappa \alpha_k \n{\nabla F(x_k)} \n{d}  - \alpha_k^2 \frac{L}{2}\n{d}^2 \leq 2L_f\varepsilon \, .
	\end{equation}
	Finally, we get
	\begin{equation}
		 \n{\nabla F(x_{k})} \leq \frac{1}{\kappa} \left( \frac{\alpha_k L\n{d_k}}{2} + \frac{2L_f\varepsilon}{\alpha_k \n{d_k} } \right)  \leq \frac{1}{\kappa} \left( \frac{b_2 \Delta_k L}{2} + \frac{2L_f\varepsilon}{b_1 \Delta_k} \right) \, .
	\end{equation}
\end{proof}

We now extend Theorem~\ref{th:smoothcompl} to our mesh based scheme. The main difference is the absence of complexity estimates, which to our knowledge are not available for MADS schemes.
\begin{theorem}
	Let Assumptions~\ref{as:fLip},~\ref{as:Flip} and~\ref{as:comp} hold. Let $\{x_k\}$ be a sequence generated by Algorithm~\ref{alg:MADS}.

		\begin{enumerate}
		\item 		If $\Alpha > 0$, then the algorithm terminates in a finite number of iterations, with the last iterate $x_{\bar{k}}$ satisfying,
		\begin{equation} \label{eq:nablaboundmads}
			\n{\nabla F(x_{\bar{k}})} \leq \frac{1}{\kappa} \left( \frac{b_2\Alpha L}{2} + \frac{2L_f\varepsilon}{\Alpha b_1} \right) \, .
		\end{equation}
		\item If, furthermore, it holds that $\Alpha = 2\sqrt{\frac{L_f \varepsilon}{b_1b_2L}}$, then
		\begin{equation} \label{eq:glowbound2}
			\n{\nabla F(x_{\bar{k}})} \leq \frac{1}{\kappa} \sqrt{L b_2 L_f \varepsilon/ b_1} \, .
		\end{equation}
		\item If $\Alpha = 0$, then $\liminf \alpha_k = 0$, and if additionally $\alpha_0 \geq \barAlpha=  2\sqrt{\frac{L_f \varepsilon}{b_1b_2L}}$, for some $\bar{k} \in \mathbb{N}_0$ we have
		\begin{equation} 
			\n{\nabla F(x_{\bar{k}})} \leq \frac{1}{\theta \kappa} \left( \frac{L\Alpha b_2}{2} + \frac{2L_f\varepsilon}{b_1\Alpha } \right) \, ,
		\end{equation}
		and 
		\begin{equation} 
			F(x_k) \leq F(x_{\bar{k}}) + 2L_f\varepsilon \quad \textnormal{ for all }k\geq \bar{k} \, .
		\end{equation}
	\end{enumerate}
\end{theorem}
\begin{proof}
  1. 	Since the frame parameter $\Delta_k$ is lower bounded, the mesh parameter $\alpha_k$ is lower bounded as well, and, by the subsequent finiteness of $\bigcup_{k \in \mathbb{N}_0} M_k$, the algorithm terminates in a finite number of iterations. By the termination criterion, at the last iteration $\bar{k}$ we have $\Delta_{\bar{k}} = \Alpha$. 
	Since the last iteration is unsuccessful, we hence get
	\begin{equation}
			\n{\nabla F(x_{\bar{k}})} \leq \frac{1}{\kappa} \left( \frac{b_2 \Delta_k L}{2} + \frac{2L_f\varepsilon}{b_1 \Delta_k} \right) \\
			= \frac{1}{\kappa} \left( \frac{b_2 \Alpha L}{2} + \frac{2L_f\varepsilon}{b_1 \Alpha} \right)	\, ,
	\end{equation}
	where we applied Lemma~\ref{l:unsucc2} in the second inequality. \\
	2. Follows from the previous point replacing $\Alpha$ with the given value in~\eqref{eq:nablaboundmads}. \\
	3. The property $\liminf \alpha_k = 0$ follows from standard arguments used in the analysis of MADS schemes, already mentioned in the proof of Lemma~\ref{l:spherelimit}. The result then follows from point 1 and 2 (similarly to point 3 in Theorem~\ref{th:smoothcompl}).
\end{proof}

\section{Numerical illustration}
\label{sec:numerical_illustration}

In this section, we evaluate the performance of the proposed algorithms on a large collection of nonlinear bilevel optimization problems.

Three direct-search solvers derived from Algorithm~\ref{alg:2} and Algorithm~\ref{alg:3} were implemented in Matlab: \textbf{Mesh-DS} (related to Algorithm~\ref{alg:3}) with the mesh defined as in~\cite[Algorithm 8.2]{audet2017derivative}, \textbf{Coordinate-DS} (related to Algorithm~\ref{alg:2}) with $D_k=[\mathcal{B}_{\oplus}, -\mathcal{B}_{\oplus}]$ (where $\mathcal{B}_{\oplus}$ is the canonical basis of $\real^n$), and \textbf{Random-DS} (related to Algorithm~\ref{alg:2}) with $D_k=[\frac{v}{\|v\|}, -\frac{v}{\|v\|}]$, where $v \in \real^n$ is a uniformly generated vector.

In our tests, the parameters used for Algorithm~\ref{alg:2} and Algorithm~\ref{alg:3} were set as follows: $\Alpha=~10^{-6}$, $\theta=\frac{1}{2}$, $\alpha_0=1$, $c=10^{-3}$, and $\gamma=2$. For all the tested approaches, the optional search step (Step 1) was not included. Instead, in the poll step, when we observed a decrease along a specific direction, we
further explored it by using a simple extrapolation strategy (i.e., we multiplied the step-size $\alpha_k$ by $\gamma$ and re-evaluated the function).

In our implementation, the  lower-level problem is solved using the \textbf{fmincon} Matlab procedure. To quantify the impact of inexact lower-level solutions on the performances, we used  2 different accuracies when solving the lower-level problem (i.e., LL\_{tol} $\in\{10^{-3}, 10^{-6}\}$). The rest of the \textbf{fmincon} default parameters were kept unchanged. A feasibility tolerance of $10^{-6}$ for constraints violation was used in the solution of the lower-level problem.

The three solvers, \textbf{Mesh-DS}, \textbf{Coordinate-DS}, and \textbf{Random-DS}, were evaluated using $33$ small-scale bilevel optimization problems from the BOLIB Matlab library~\cite{BOLIB_2020}. This library consists of a collection of academic and real-world problems. The dimensions of the tested instances, with respect to the upper-level problem, do not exceed 10 variables. Since an initial point is not provided, we generated five problem instances by randomly selecting five different initial points, thus getting a total of $175$ problem instances.

The computational analysis is carried out by using well-known tools from the literature, that is  data and performance profiles~(see,e.g.,~\cite{JJMore_SMWild_2008} for further details). 
We briefly recall here their definitions.	Given a set~$S$ of algorithms and a set~$P$ of problems, for $s\in S$ and $p \in P$, let $t_{p,s}$ be the number of function evaluations required by algorithm $s$ on problem~$p$ to satisfy the condition 

     \begin{equation}\label{eq:stop}
     	\tilde F(x_k) \leq \tilde F_{\mbox{low}} + \alpha(\tilde F(x_0) - \tilde F_{\mbox{low}})\, , 
     \end{equation}
     
     where $\gamma_p \in (0, 1)$ and $\tilde F_{\mbox{low}}$ is the best objective function value achieved by any solver on problem~$p$. Then, the performance and data profiles of solver $s$ are defined by
     \begin{eqnarray*}
     	\rho_s(\gamma) & = & \frac{1}{|P|}\left|\left\{p\in P: \frac{t_{p,s}}{\min\{t_{p,s'}:s'\in S\}}\leq\gamma\right\}\right|,\\
     	d_s(\kappa) & = & \frac{1}{|P|}\left|\left\{p\in P: t_{p,s}\leq\kappa(n_p+1)\right\}\right|\, ,
     \end{eqnarray*}
     where $n_p$ is the dimension of problem $p$.
     We used a budget of $500$ upper level function evaluations in our experiments.


\begin{figure}[h]
	\centering
		\includegraphics[scale=0.46]{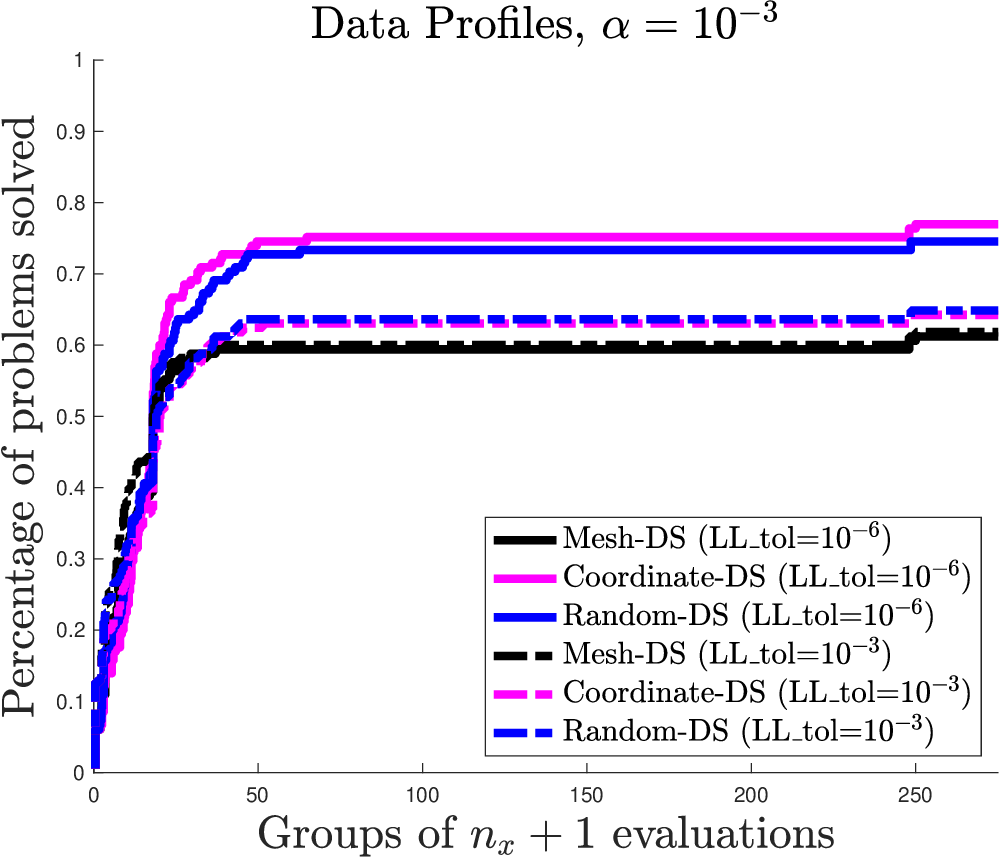}
  	\includegraphics[scale=0.46]{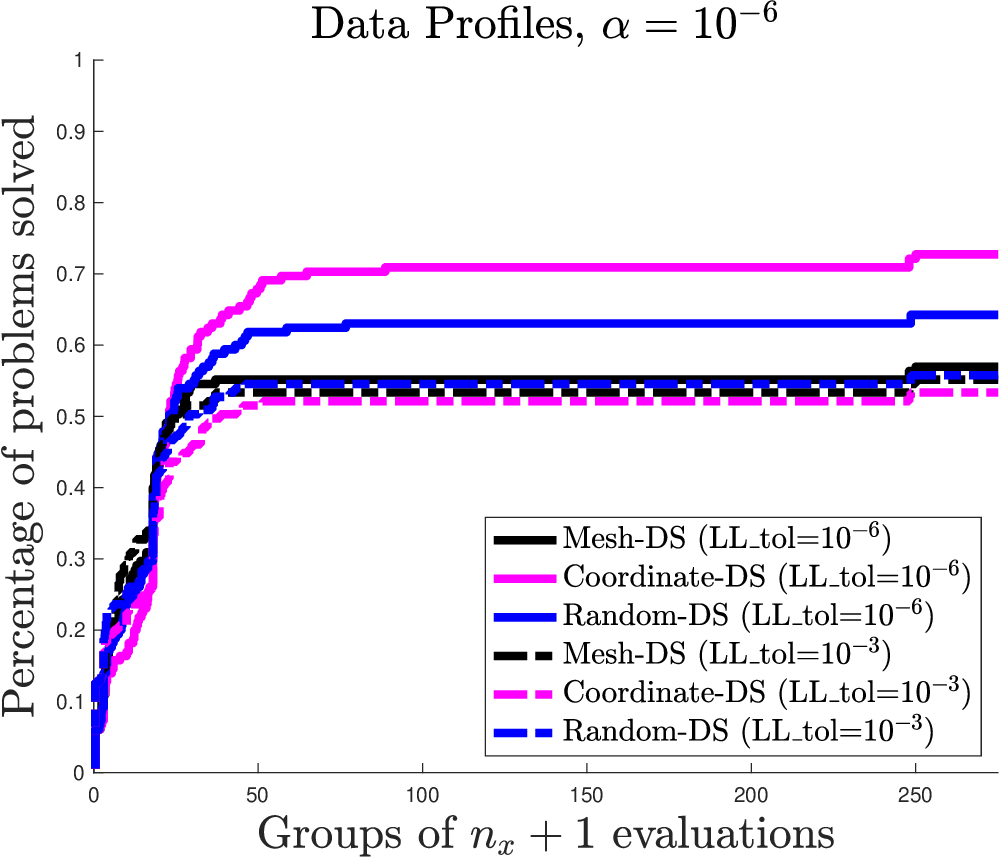}
		\caption{Data profiles using two  type of tolerances to get an approximate minimizer for the lower-level problem.}
		\label{fig:1}
\end{figure}

\begin{figure}[h]
	\centering
		\includegraphics[scale=0.46]{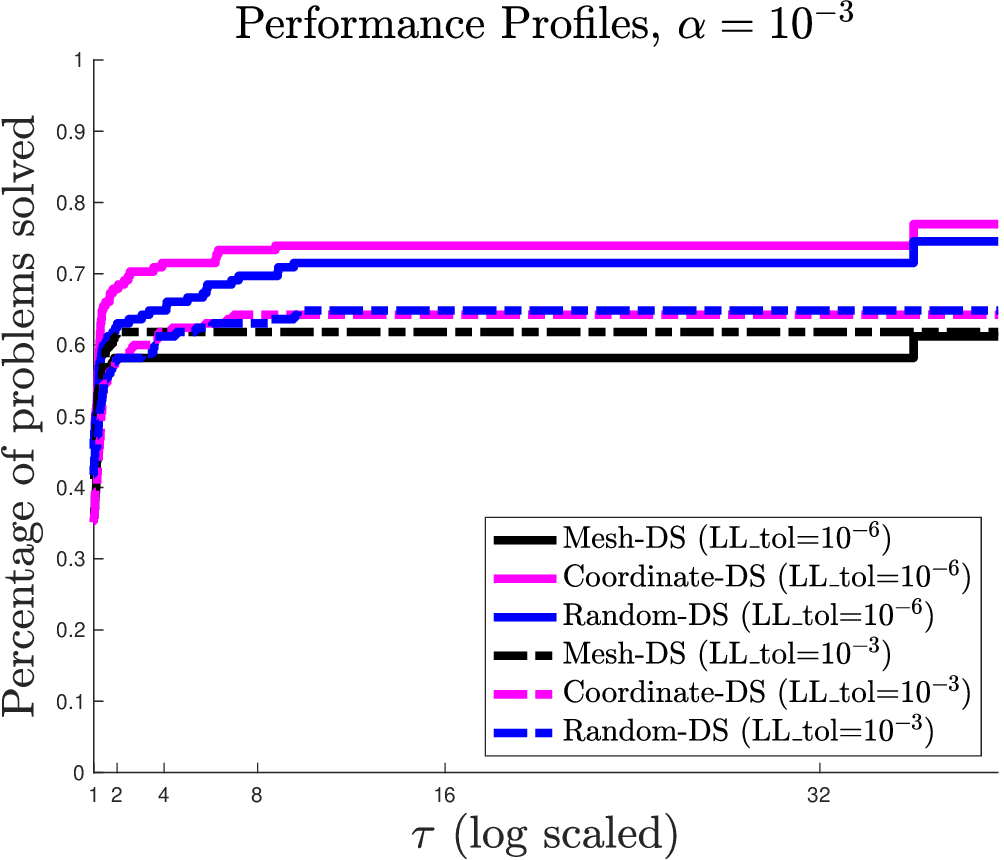}
  	\includegraphics[scale=0.46]{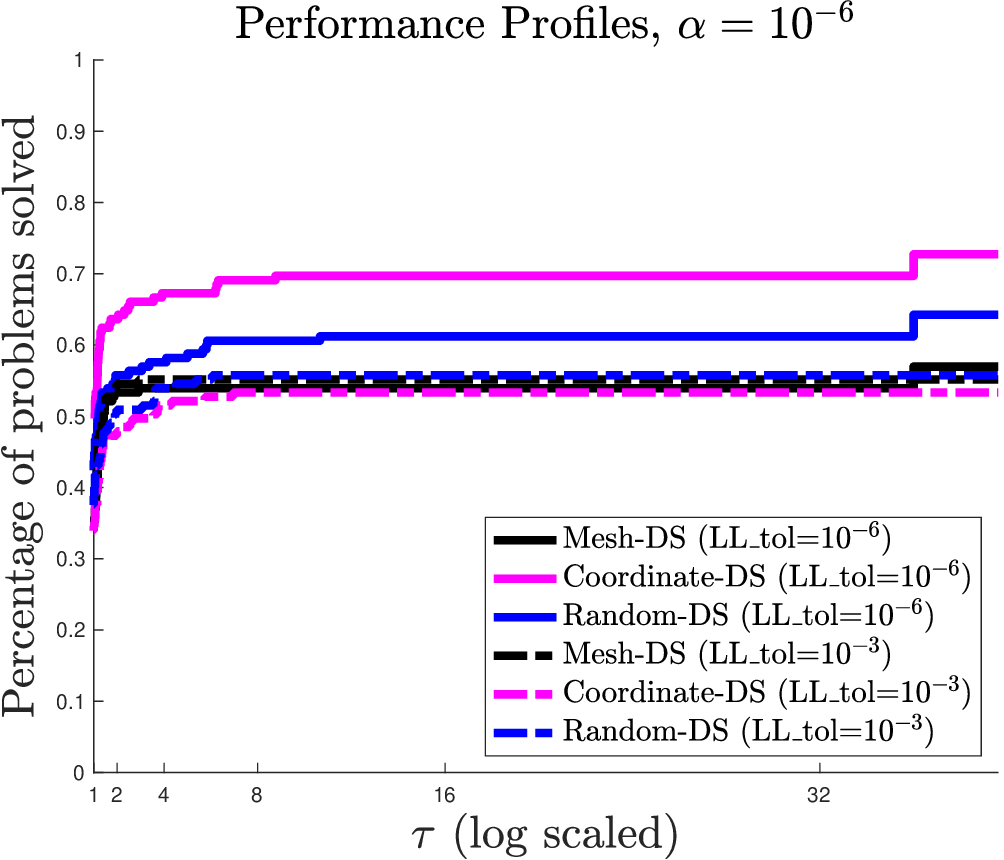}
		\caption{Performance profiles using two  type of tolerances to get an approximate minimizer for the lower-level problem.}
		\label{fig:2}
\end{figure}

Figures~\ref{fig:1}-\ref{fig:2} depict the resulting performance and data profiles, respectively, considering two levels of accuracy $\alpha$: $10^{-3}$ and $10^{-6}$. From Figure~\ref{fig:2}, it can be observed that the \textbf{Coordinate-DS} approach performs the best in terms of both efficiency (i.e., $\tau=1$) and robustness (i.e., larger $\tau$), particularly when the lower problem is solved accurately (i.e., LL\_tol=$10^{-6}$). The data profiles (see Figure~\ref{fig:2}) indicate that all the direct-search approaches perform similarly for small budgets. However, as the budget increases, the accuracy of the lower problem becomes impactful on the solver's performance. Overall, on the tested problems, the directional direct-search approaches seem to outperform the mesh-based direct-search approach.

\section{Conclusion}
In this work, we proposed an inexact direct-search based algorithmic framework  for bilevel optimization, under the assumption that the lower-level problem can be solved within a fixed accuracy. We then proved convergence of two different classes of methods fitting our scheme, that is directional direct-search methods with sufficient decrease and mesh based schemes with simple decrease. Our results include complexity estimates for a directional direct-search scheme tailored for BO with smooth true objective, which extends previously known complexity estimates for the single level case. We also 
considered  the nonsmooth case and gave convergence guarantees to  $(\delta, \epsilon)$-Goldstein stationary points for both classes, thus nicely extending the known Clarke stationary point convergence properties of analogous schemes in the single level case. A lower bound on the stepsize allows these method to convergence to a point with the desired stationarity properties in a finite number of iterations. Preliminary numerical results suggest that directional  direct-search methods might lead to better performance than mesh based strategies in this context. \\ 
Future developments include the extensions of our algorithms to constrained and stochastic objectives, as well as numerical comparisons with recent zeroth order smoothing based approaches for BO. \\

\textbf{Data availability.} The data analysed during the current study are available in the BOLIB library and the code will be made available by the authors upon reasonable request.

\clearpage
\bibliographystyle{plain}
\bibliography{refs}
\end{document}